\documentclass[12pt]{article}

\usepackage[T1]{fontenc}
\usepackage[utf8]{inputenc}
\usepackage[english]{babel}
\usepackage{amsmath,amssymb,amsfonts,amsthm}
\usepackage[scr=rsfs]{mathalpha}
\usepackage{geometry}
\usepackage{microtype}
\usepackage[numbers,sort&compress]{natbib}
\usepackage{mathtools}
\mathtoolsset{showonlyrefs=true}
\usepackage[colorlinks=true,linkcolor=blue,citecolor=blue,urlcolor=blue]{hyperref}

\numberwithin{equation}{section}

\newtheorem{theorem}{Theorem}[section]
\newtheorem{proposition}[theorem]{Proposition}
\newtheorem{lemma}[theorem]{Lemma}

\newtheorem{definition}[theorem]{Definition}
\theoremstyle{remark}
\newtheorem{remark}[theorem]{Remark}

\newcommand{\Pcal}{\mathbb P}

\newcommand{\Lop}{\mathscr L}
\newcommand{\Lag}{\mathcal L}
\newcommand{\Sfun}{\mathcal S}
\newcommand{\dd}{\,\mathrm d}

\title{Hahn-Type Orthogonality for the Resolvent Delta Operator
\texorpdfstring{\(\mathscr L=D(I-D)^{-1}\)}{L=D(I-D)^(-1)}}

\author{%
Youssef Esstafa\textsuperscript{1}
\quad and \quad
Ridha Sfaxi\textsuperscript{2}
\\[1.2ex]
\small \textsuperscript{1}Le Mans Universit\'e, Laboratoire Manceau de Math\'ematiques,\\
\small Avenue Olivier Messiaen, 72085 Le Mans Cedex 09, France.\\
\small Email: \href{mailto:youssef.esstafa@univ-lemans.fr}{\texttt{youssef.esstafa@univ-lemans.fr}}
\\[2ex]
\small \textsuperscript{2}University of Gabes, Faculty of Sciences of Gabes,\\
\small City Erriadh 6072 Zrig, Gabes, Tunisia.\\
\small Email: \href{mailto:ridhasfaxi@gmail.com}{\texttt{ridhasfaxi@gmail.com}}
}

\date{}

\begin{document}
\maketitle

\begin{abstract}
We investigate orthogonality under the shift-invariant resolvent delta operator
\(\mathscr L=D(I-D)^{-1}\). First, we prove a rigidity theorem: a monic
orthogonal polynomial sequence is \(\mathscr L\)-Appell if and only if it is a
translated Laguerre sequence. We then move beyond the Appell setting and show
that the diagonal family associated with
\[
   \mathcal S_\beta
   =\frac{\beta}{\beta+1}\mathcal L^0+\frac{1}{\beta+1}\delta
\]
satisfies the exact Hahn relation
\[
   \mathscr L B_{n+1}(x;\beta)=(n+1)B_n(x;\beta+1).
\]
A second rigidity result shows that a single pointwise relation
\(P_{n+1}(c)=\theta\,\mathscr L P_{n+1}(c)\) already forces the orthogonality
functional to be a translated Laguerre functional or a translated member of
the diagonal family. Connection and structure relations, a semiclassical
Pearson equation, a differential characterization, and a generating function
complete the description. Thus the resolvent operator provides an explicit
setting in which delta-operator lowering and semiclassical orthogonality are
linked by exact structural identities.
\end{abstract}



\section{Introduction}

One of the central themes in the theory of orthogonal polynomials is the
rigidity produced by lowering operators. For the ordinary derivative, Hahn's
theorem \cite{Hahn1935}, closely related to Bochner's characterization by
second-order differential equations \cite{Bochner1929}, shows that requiring a
derived sequence to remain orthogonal forces the original sequence into the
classical framework. This principle connects recurrence relations, Pearson
equations, Rodrigues formulae and differential equations, and it remains a
basic source of classification problems in the subject; see
\cite{Chihara1978,Szego1975}.

A parallel algebraic viewpoint is provided by Sheffer sequences and finite
operator calculus. Meixner's classification \cite{Meixner1934} already
demonstrates how restrictive orthogonality is inside the Sheffer class; a
modern account of its place in the theory is given by Butzer and Koornwinder
\cite{ButzerKoornwinder2019}. Sheffer's work \cite{Sheffer1939} and the finite
operator calculus of Rota, Kahaner and Odlyzko
\cite{RotaKahanerOdlyzko1973}, later systematized by Roman
\cite{Roman1984}, place such questions in the language of delta operators:
shift-invariant operators on the polynomial algebra which lower degree by one
and do not annihilate \(x\). The classical Laguerre calculus is already
visible in this framework. In particular, the operator \(D/(D-1)\), differing
only by a sign from the operator considered below, is the standard
Laguerre-type delta operator in finite operator calculus. What is not supplied
by the formal calculus itself is a Hahn theory: one must determine when
orthogonality is preserved by the lowering operation and what structural
constraints this imposes on the orthogonality functional.

We study this problem for the resolvent delta operator
\begin{equation}\label{eq:def-L-intro}
   \Lop=D(I-D)^{-1}=(I-D)^{-1}D,
\end{equation}
where \(D=d/dx\). Since \(D\) is locally nilpotent on the polynomial algebra
\(\Pcal\),
\[
   (I-D)^{-1}=\sum_{j\ge0}D^j
\]
is a finite sum on each polynomial. Hence \(\Lop\) is well defined,
shift-invariant, lowers degree by one and satisfies \(\Lop x=1\). Our aim is
not to introduce a new symbolic calculus, but to develop the orthogonality
theory specific to this resolvent operator. The algebraic framework of
regular linear functionals due to Maroni \cite{Maroni1991} is particularly
well suited to this purpose because it allows lowering identities, spectral
transformations and Pearson equations to be treated on the same footing. We
use the normalization of Laguerre functionals adopted in
\cite{MaroniTI1994}.

The first phenomenon is a complete rigidity result. For the monic Laguerre
polynomials
\[
   \hat l_n(x;\alpha)=(-1)^n n!L_n^{(\alpha)}(x),
   \qquad \alpha\notin\{-1,-2,\ldots\},
\]
the standard identities
\begin{equation}\label{eq:laguerre-classical-relations-intro}
  D\hat l_{n+1}(x;\alpha)=(n+1)\hat l_n(x;\alpha+1),
  \qquad
  \hat l_n(x;\alpha)=\hat l_n(x;\alpha+1)
      +n\hat l_{n-1}(x;\alpha+1)
\end{equation}
imply
\begin{equation}\label{eq:L-laguerre-intro}
   \Lop\hat l_{n+1}(x;\alpha)=(n+1)\hat l_n(x;\alpha).
\end{equation}
Thus Laguerre polynomials are \(\Lop\)-Appell. We prove the converse:
\emph{among all monic orthogonal polynomial sequences, the
\(\Lop\)-Appell property characterizes the Laguerre family up to translation}.
The proof is direct: the \(\Lop\)-Appell relation fixes the exponential
generating function up to one scalar series, and Favard's three-term
recurrence forces that series to have exactly the Laguerre form. This gives a
self-contained orthogonality classification for the fixed delta operator
\(\Lop\).

The genuinely Hahn-type phenomenon appears when the derived sequence is
required to be orthogonal but not necessarily identical to the original one.
For this purpose we consider the diagonal family introduced in the framework
of generalized coherent pairs by Maroni and Sfaxi
\cite{MaroniSfaxi2000}. It is orthogonal with respect to
\begin{equation}\label{eq:S-beta-intro}
   \Sfun_\beta
   =
   \frac{\beta}{\beta+1}\Lag^0
   +
   \frac{1}{\beta+1}\delta,
   \qquad
   \beta\notin\{0,-1,-2,\ldots\},
\end{equation}
where \(\Lag^0\) is the normalized Laguerre functional of parameter zero and
\(\delta\) is the Dirac functional at the origin. If
\(\{B_n(x;\beta)\}_{n\ge0}\) denotes the corresponding MOPS, our main lowering
identity is
\begin{equation}\label{eq:B-lowering-intro}
   \Lop B_{n+1}(x;\beta)=(n+1)B_n(x;\beta+1),
   \qquad n\ge0.
\end{equation}
Thus the first \(\Lop\)-derived sequence is again orthogonal, but with a
non-trivial parameter shift. This is the central mechanism of the paper: the
point-mass perturbation in \(\Sfun_\beta\) is exactly compatible with the
resolvent lowering operator.

This mechanism is accompanied by a second rigidity statement. We prove that,
for a MOPS with normalized regular functional \(u_0\), the single boundary
condition
\[
   P_{n+1}(c)=\theta\,\Lop P_{n+1}(c),\qquad n\ge0,
\]
forces \(u_0\) to be either a translated Laguerre functional \(\tau_c\Lag^0\)
or a translated diagonal functional
\(\tau_c\Sfun_{-\theta/(\theta+1)}\), subject only to the natural regularity
restrictions on \(\theta\). Unlike the \(\Lop\)-Appell theorem, this result
does not assume that the lowered sequence coincides with the original one.
It therefore gives a precise partial classification under a minimal
pointwise compatibility condition and explains why the diagonal family
emerges naturally from the resolvent calculus.

The rest of the analysis shows that this lowering relation is not an isolated
identity. We derive finite connection formulae with Laguerre polynomials,
three structure relations coupling the original and \(\Lop\)-derived
sequences, an ordinary differential Pearson equation identifying
\(\Sfun_\beta\) as a semiclassical functional of class one, a corresponding
differential characterization, and an exponential generating function whose
\(\Lop\)-action reproduces the parameter shift. These results connect the
delta-operator viewpoint with the functional and semiclassical approaches to
orthogonality. In this direction we use the work of Marcellán and Sfaxi on
weakly regular and semiclassical functionals
\cite{MarcellanSfaxi2007a,MarcellanSfaxi2007}, as well as related structural
results of Sfaxi and Alaya \cite{SfaxiAlaya2006}. The broader relation with
Sobolev and coherent-pair constructions is discussed in
\cite{Iserles1991,AlfaroMarcellanRezola1993}, while lowering operators adapted
to Laguerre-type structures were studied by Aloui, Marcellán and Sfaxi
\cite{Aloui2013}.

The paper is organized as follows. Section~\ref{sec:prelim} introduces the
algebraic setting and the action of \(\Lop\) on polynomials and linear
functionals. Section~\ref{sec:laguerre} proves the rigidity of orthogonal
\(\Lop\)-Appell sequences and the associated functional identities.
Section~\ref{sec:B} develops the diagonal family, proves the exact Hahn
lowering relation \eqref{eq:B-lowering-intro}, establishes the boundary
rigidity theorem, and derives the structure, semiclassical and generating
function characterizations. The final section summarizes the resulting orthogonality picture and the two
rigidity mechanisms established in the paper.

\section{Algebraic setting and the resolvent operator}\label{sec:prelim}

Let \(\Pcal\) be the vector space of polynomials with complex coefficients and let \(\Pcal'\) be its algebraic dual.  The action of \(u\in\Pcal'\) on \(p\in\Pcal\) is denoted by \(\langle u,p\rangle\), and \((u)_n=\langle u,x^n\rangle\) is the \(n\)-th moment.  We shall use the standard operations
\begin{align*}
  \langle pu,q\rangle&=\langle u,pq\rangle,\qquad
  \langle Du,q\rangle=-\langle u,Dq\rangle,\qquad
  \langle\delta_c,q\rangle=q(c),  \\
  \langle\tau_cu,q\rangle&=\langle u,q(x+c)\rangle,
\end{align*}
where \(p,q\in\Pcal\), \(u\in\Pcal'\), \(c\in\mathbb C\), and \(\delta=\delta_0\).  A linear functional is normalized if \((u)_0=1\).

A monic polynomial sequence, abbreviated MPS, is a sequence \(\{P_n\}_{n\ge0}\) with \(\deg P_n=n\) and leading coefficient one.  Its dual sequence \(\{u_n\}_{n\ge0}\subset\Pcal'\) is determined by
\[
   \langle u_n,P_m\rangle=\delta_{nm},\qquad n,m\ge0.
\]
The sequence is a monic orthogonal polynomial sequence, abbreviated MOPS, if there exists a regular linear functional \(u_0\) such that
\[
   \langle u_0,P_nP_m\rangle=0\quad(n\ne m),
   \qquad
   \langle u_0,P_n^2\rangle\ne0\quad(n\ge0).
\]
Equivalently, \(\{P_n\}_{n\ge0}\) satisfies a three-term recurrence
\begin{equation}\label{eq:TTRR-general}
\begin{cases}
P_{n+1}(x)=(x-\beta_n)P_n(x)-\gamma_nP_{n-1}(x),\qquad n\ge1,\\
P_0(x)=1,
\qquad P_1(x)=x-\beta_0,
\end{cases}
\end{equation}
with \(\beta_n\in\mathbb C\) and \(\gamma_n\ne0\).  When \(\{P_n\}\) is orthogonal with respect to \(u_0\), its dual sequence is
\begin{equation}\label{eq:dual-orthogonal}
   u_n=\frac{P_nu_0}{\langle u_0,P_n^2\rangle},
   \qquad n\ge0.
\end{equation}
We shall also use the derivative sequence
\begin{equation}\label{eq:derivative-sequence}
   P_n^{[1]}(x)=\frac{1}{n+1}DP_{n+1}(x),
   \qquad n\ge0.
\end{equation}
If \(\{u_n^{[1]}\}\) denotes its dual sequence, then
\begin{equation}\label{eq:dual-derivative}
   Du_n^{[1]}=-(n+1)u_{n+1},
   \qquad n\ge0.
\end{equation}

We recall the monic Laguerre polynomial sequence \(\{\hat l_n(x;\alpha)\}_{n\ge0}\), with
\(\alpha\notin\{-1,-2,\ldots\}\).  It is orthogonal with respect to the normalized Laguerre functional \(\Lag^\alpha\), whose moments are
\begin{equation}\label{eq:Lag-moments}
   (\Lag^\alpha)_n=\frac{\Gamma(n+\alpha+1)}{\Gamma(\alpha+1)},
   \qquad n\ge0.
\end{equation}
For \(\alpha>-1\),
\begin{equation}\label{eq:Lag-integral}
   \langle\Lag^\alpha,p\rangle=\frac{1}{\Gamma(\alpha+1)}
   \int_0^\infty x^\alpha e^{-x}p(x)\dd x.
\end{equation}
The functional equation of \(\Lag^\alpha\) is
\begin{equation}\label{eq:Lag-Pearson-D}
   (x\Lag^\alpha)' +(x-\alpha-1)\Lag^\alpha=0.
\end{equation}
The monic polynomials satisfy
\begin{equation}\label{eq:Lag-recurrence}
\begin{cases}
\hat l_{n+1}(x;\alpha)=\bigl(x-(2n+\alpha+1)\bigr)\hat l_n(x;\alpha)
     -n(n+\alpha)\hat l_{n-1}(x;\alpha),\quad n\ge1,\\
\hat l_0(x;\alpha)=1,
\qquad \hat l_1(x;\alpha)=x-\alpha-1,
\end{cases}
\end{equation}
and their squared norms are
\begin{equation}\label{eq:Lag-norm}
   \theta_n(\alpha):=\langle\Lag^\alpha,\hat l_n(x;\alpha)^2\rangle
   =n!\frac{\Gamma(n+\alpha+1)}{\Gamma(\alpha+1)}.
\end{equation}
The dual sequence associated with \(\{\hat l_n(x;\alpha)\}\) is therefore
\begin{equation}\label{eq:Lag-dual}
   \Lag_n^\alpha=\theta_n(\alpha)^{-1}\hat l_n(x;\alpha)\Lag^\alpha.
\end{equation}
We shall repeatedly use the elementary identities
\begin{align}
D\hat l_{n+1}(x;\alpha)&=(n+1)\hat l_n(x;\alpha+1),\label{eq:Lag-D}\\
\hat l_n(x;\alpha)&=\hat l_n(x;\alpha+1)+n\hat l_{n-1}(x;\alpha+1),\label{eq:Lag-shift}\\
x\hat l_n(x;\alpha+1)&=\hat l_{n+1}(x;\alpha)+(n+\alpha+1)\hat l_n(x;\alpha),\label{eq:Lag-x}\\
xD^2\hat l_n(x;\alpha)&-(x-\alpha-1)D\hat l_n(x;\alpha)=-n\hat l_n(x;\alpha).\label{eq:Lag-ODE}
\end{align}
and the explicit value
\begin{equation}\label{eq:Lag-at-zero}
   \hat l_n(0;\alpha)=(-1)^n\frac{\Gamma(n+\alpha+1)}{\Gamma(\alpha+1)}.
\end{equation}

We now record the elementary properties of \(\Lop\).  Since \(D\) is locally nilpotent on \(\Pcal\), the inverse \((I-D)^{-1}\) is the finite sum \(\sum_{j\ge0}D^j\) on every polynomial.  Hence
\begin{equation}\label{eq:L-series}
   \Lop p=\sum_{j\ge1}D^jp,
   \qquad p\in\Pcal.
\end{equation}
In particular,
\begin{equation}\label{eq:L-monomial}
   \Lop(x^n)=\sum_{k=0}^{n-1}\frac{n!}{k!}x^k
   =nx^{n-1}+n(n-1)x^{n-2}+\cdots+n!,
   \qquad n\ge1.
\end{equation}
Thus \(\Lop\) lowers degree by one and \(\Lop x=1\).  Moreover, \(\Lop\) commutes with \(D\), and consequently with every translation.

On \(\Pcal'\) we use the dual convention
\begin{equation}\label{eq:L-duality}
   \langle\Lop u,p\rangle=-\langle u,\Lop p\rangle,
   \qquad u\in\Pcal',\ p\in\Pcal.
\end{equation}
Combining \eqref{eq:L-series} with \(\langle D^ju,p\rangle=(-1)^j\langle u,D^jp\rangle\), one obtains the useful expansion
\begin{equation}\label{eq:L-dual-series}
   \Lop u=\sum_{j\ge1}(-1)^{j+1}D^ju
   =Du-D^2u+D^3u-\cdots.
\end{equation}
The map \(\Lop:\Pcal\to\Pcal\) is onto.  Indeed, it sends \(x^{n+1}/(n+1)\) to \(x^n\) plus lower degree terms.  It follows from \eqref{eq:L-duality} that \(\Lop\) is injective on \(\Pcal'\): if \(\Lop u=0\), then \(\langle u,\Lop p\rangle=0\) for every \(p\), hence \(u=0\).

\begin{lemma}\label{lem:L-product}
For every \(p\in\Pcal\) and every \(u\in\Pcal'\),
\begin{align}
   \Lop(xp)&=x\Lop(p)+(I+\Lop)^2p,\label{eq:L-product-poly}\\
   \Lop(xu)&=x\Lop(u)+(I-\Lop)^2u.\label{eq:L-product-dual}
\end{align}
\end{lemma}

\begin{proof}
Since \(Dxp=xDp+p\), we compute directly from \(\Lop=D(I-D)^{-1}\).  The commutator identity
\([D,x]=1\) gives
\[
   (I-D)x=x(I-D)-1.
\]
Multiplying by \((I-D)^{-1}\) on the left and on the right yields
\[
   (I-D)^{-1}x=x(I-D)^{-1}+(I-D)^{-2}.
\]
Consequently,
\[
\begin{aligned}
 \Lop(xp)
   &=D(I-D)^{-1}(xp)                                      \\
   &=D\{x(I-D)^{-1}p+(I-D)^{-2}p\}                         \\
   &=xD(I-D)^{-1}p+(I-D)^{-1}p+D(I-D)^{-2}p                 \\
   &=x\Lop p+\bigl(I+\Lop\bigr)^2p,
\end{aligned}
\]
because \((I-D)^{-1}=I+\Lop\).  This proves \eqref{eq:L-product-poly}.  For the dual identity, let \(q\in\Pcal\).  Using \eqref{eq:L-duality} and \eqref{eq:L-product-poly},
\[
\begin{aligned}
\langle\Lop(xu)-x\Lop u,q\rangle
   &=-\langle xu,\Lop q\rangle+\langle u,\Lop(xq)\rangle       \\
   &=\langle u,\Lop(xq)-x\Lop q\rangle                         \\
   &=\langle u,(I+\Lop)^2q\rangle                               \\
   &=\langle (I-\Lop)^2u,q\rangle.
\end{aligned}
\]
This is \eqref{eq:L-product-dual}.
\end{proof}

The following identity is a useful bridge between \(\Lop\) and the Laguerre functional at the origin.

\begin{lemma}\label{lem:Ldelta}
One has
\begin{equation}\label{eq:L0-delta}
   \Lag^0=\delta-\Lop\delta.
\end{equation}
Consequently, for every polynomial \(p\) and every \(t\in\mathbb R\),
\begin{equation}\label{eq:L-integral-representation}
   \Lop p(t)=e^t\int_t^\infty e^{-x}p(x)\dd x-p(t)
   =e^t\int_t^\infty e^{-x}\bigl(p(x)-p(t)\bigr)\dd x.
\end{equation}
\end{lemma}

\begin{proof}
It is enough to test the functional \(\delta-\Lop\delta\) on the Laguerre basis \(\{\hat l_n(x;0)\}\).  For \(n=0\), both sides have mass one.  For \(n\ge1\), \eqref{eq:Lag-D} and \eqref{eq:Lag-shift}, with \(\alpha=0\), imply
\[
   \Lop\hat l_n(x;0)=n\hat l_{n-1}(x;0).
\]
Using this identity, \eqref{eq:L-duality} and \eqref{eq:Lag-at-zero}, we get
\[
\begin{aligned}
\langle\delta-\Lop\delta,\hat l_n(x;0)\rangle
   &=\hat l_n(0;0)+\langle\delta,\Lop\hat l_n(x;0)\rangle  \\
   &=(-1)^nn!+n\hat l_{n-1}(0;0)=0.
\end{aligned}
\]
Thus \(\delta-\Lop\delta\) is the normalized functional annihilating all \(\hat l_n(x;0)\), \(n\ge1\), namely \(\Lag^0\).

For the integral formula, apply \eqref{eq:L0-delta} to the translated polynomial \(x\mapsto p(x+t)\).  Since \(\Lop\) commutes with translations,
\[
   \int_0^\infty e^{-x}p(x+t)\dd x=p(t)+\Lop p(t).
\]
Changing variables gives the first equality in \eqref{eq:L-integral-representation}; the second follows from \(e^t\int_t^\infty e^{-x}\dd x=1\).
\end{proof}

\section{The Laguerre family and \texorpdfstring{\(\mathscr L\)}{L}-classicality}\label{sec:laguerre}

The identities recalled above immediately imply that Laguerre polynomials are lowered by \(\Lop\).

\begin{proposition}\label{prop:Lag-lowering}
For every \(\alpha\notin\{-1,-2,\ldots\}\),
\begin{equation}\label{eq:Lag-L-lowering}
   \Lop\hat l_{n+1}(x;\alpha)=(n+1)\hat l_n(x;\alpha),
   \qquad n\ge0.
\end{equation}
Moreover,
\begin{align}
   \Lop\bigl(x\hat l_n(x;1)\bigr)&=(n+1)x\hat l_{n-1}(x;1),
      \qquad n\ge1,\label{eq:L-xLag1}\\
   D^2\bigl(x\hat l_n(x;1)\bigr)&=(n+1)\Lop\hat l_n(x;1),
      \qquad n\ge0.\label{eq:D2-xLag1}
\end{align}
\end{proposition}

\begin{proof}
Combining \eqref{eq:Lag-D} and \eqref{eq:Lag-shift}, we obtain
\[
   D\hat l_{n+1}(x;\alpha)=(n+1)(I-D)\hat l_n(x;\alpha).
\]
Applying \((I-D)^{-1}\) gives \eqref{eq:Lag-L-lowering}.  To prove \eqref{eq:L-xLag1}, use the recurrence \eqref{eq:Lag-recurrence} with \(\alpha=1\):
\[
   x\hat l_n(x;1)=\hat l_{n+1}(x;1)+2(n+1)\hat l_n(x;1)+n(n+1)\hat l_{n-1}(x;1).
\]
Applying \eqref{eq:Lag-L-lowering} to each term gives \eqref{eq:L-xLag1}.  Finally, \eqref{eq:D2-xLag1} follows from
\[
   D^2(x\hat l_n(x;1))=(n+1)n\hat l_{n-1}(x;1),
\]
after differentiating the identity \(D(x\hat l_n(x;1))=(n+1)\hat l_n(x;0)\), and from \eqref{eq:Lag-L-lowering} with \(\alpha=1\).
\end{proof}

\begin{definition}
A monic polynomial sequence \(\{P_n\}_{n\ge0}\) is called \(\Lop\)-Appell if
\begin{equation}\label{eq:L-Appell-def}
   \Lop P_{n+1}=(n+1)P_n,
   \qquad n\ge0.
\end{equation}
It is called \(\Lop\)-Hahn admissible if \(\{P_n\}\) is a MOPS and the sequence
\begin{equation}\label{eq:L-Hahn-derived}
   P_n^{\langle1\rangle}(x)=\frac{1}{n+1}\Lop P_{n+1}(x),
   \qquad n\ge0,
\end{equation}
is also a MOPS.
\end{definition}

\begin{theorem}\label{thm:Lag-Appell-characterization}
Let \(\{P_n\}_{n\ge0}\) be a MOPS.  If \(\{P_n\}\) is \(\Lop\)-Appell, then there exist \(c\in\mathbb C\) and \(\alpha\notin\{-1,-2,\ldots\}\) such that
\begin{equation}\label{eq:Lag-translated}
   P_n(x)=\hat l_n(x-c;\alpha),
   \qquad n\ge0.
\end{equation}
Conversely, every translated Laguerre sequence \eqref{eq:Lag-translated} is \(\Lop\)-Appell and orthogonal.
\end{theorem}

\begin{proof}
Assume first that \(\{P_n\}_{n\ge0}\) is \(\Lop\)-Appell. Introduce the formal exponential generating function
\[
   F(x,t)=\sum_{n\ge0}P_n(x)\frac{t^n}{n!}.
\]
The relation \(\Lop P_{n+1}=(n+1)P_n\) is equivalent to
\[
   \Lop_x F(x,t)=tF(x,t).
\]
Since \(\Lop=D(I-D)^{-1}\), applying \(I-D\) to both sides gives
\[
   D_xF(x,t)=t(I-D_x)F(x,t).
\]
Thus
\[
   (1+t)D_xF(x,t)=tF(x,t),
\]
and therefore
\[
   F(x,t)=A(t)\exp\left(\frac{xt}{1+t}\right),
   \qquad A(0)=1.
\]

We now use orthogonality. Since \(\{P_n\}_{n\ge0}\) is a MOPS, it satisfies a
three-term recurrence
\[
   xP_n(x)=P_{n+1}(x)+\beta_nP_n(x)+\gamma_nP_{n-1}(x),
   \qquad n\ge0,
\]
with the convention \(P_{-1}=0\). Set
\[
   a(t)=(1+t)^2\frac{A'(t)}{A(t)}
   =\sum_{k\ge0}a_k\frac{t^k}{k!}.
\]
Since
\[
   \frac{\partial F}{\partial t}
   =
   \left(\frac{A'(t)}{A(t)}+\frac{x}{(1+t)^2}\right)F,
\]
we obtain
\[
   xF(x,t)=(1+t)^2\frac{\partial F}{\partial t}(x,t)-a(t)F(x,t).
\]
Comparing the coefficient of \(t^n/n!\) yields
\[
   xP_n
   =
   P_{n+1}
   +(2n-a_0)P_n
   +\bigl(n(n-1)-na_1\bigr)P_{n-1}
   -
   \sum_{k=2}^{n}\binom{n}{k}a_kP_{n-k}.
\]
Since \(\{P_m\}_{m\ge0}\) is a polynomial basis, the expansion of \(xP_n\)
in this basis is unique. Comparing it with the three-term recurrence shows
that the coefficients of \(P_{n-k}\), \(k\ge2\), must vanish. Hence
\(a_k=0\) for every \(k\ge2\). Consequently
\[
   a(t)=a_0+a_1t.
\]
It follows that
\[
   \frac{A'(t)}{A(t)}
   =
   \frac{a_0+a_1t}{(1+t)^2},
\]
and integration, together with \(A(0)=1\), gives
\[
   A(t)
   =
   (1+t)^{a_1}
   \exp\left(\frac{(a_0-a_1)t}{1+t}\right).
\]
Therefore
\[
   F(x,t)
   =
   (1+t)^{a_1}
   \exp\left(\frac{(x+a_0-a_1)t}{1+t}\right).
\]
Put
\[
   \alpha=-a_1-1,
   \qquad
   c=a_1-a_0.
\]
Then
\[
   F(x,t)
   =
   (1+t)^{-\alpha-1}
   \exp\left(\frac{(x-c)t}{1+t}\right),
\]
which is exactly the exponential generating function of the translated monic
Laguerre sequence
\[
   \hat l_n(x-c;\alpha).
\]
Thus
\[
   P_n(x)=\hat l_n(x-c;\alpha),
   \qquad n\ge0.
\]
Finally, the recurrence coefficient obtained above is
\[
   \gamma_n=n(n+\alpha),
\]
and the regularity of the MOPS implies
\[
   \alpha\notin\{-1,-2,\ldots\}.
\]
This proves the direct implication. Conversely, every translated Laguerre
sequence satisfies the \(\Lop\)-Appell relation by Proposition~\ref{prop:Lag-lowering}
and by the fact that \(\Lop\) commutes with translations. Orthogonality is
preserved by translating the corresponding regular functional. The proof is
complete.
\end{proof}

The preceding theorem clarifies the Appell part of the theory.  We now express the corresponding functional equation directly in the \(\Lop\)-calculus.

\begin{proposition}\label{prop:L-Pearson-Lag}
For \(\alpha\notin\{-1,-2,\ldots\}\), the Laguerre functional satisfies
\begin{equation}\label{eq:L-Pearson-Lag}
   \Lop(\Lag^\alpha)+\left(\frac{x}{\alpha+1}-1\right)\Lag^\alpha=0.
\end{equation}
Equivalently,
\begin{equation}\label{eq:L-Rodrigues-functional}
   \Lop^n(\Lag^\alpha)=(-1)^n\frac{\Gamma(\alpha+1)}{\Gamma(n+\alpha+1)}
   \hat l_n(x;\alpha)\Lag^\alpha,
   \qquad n\ge0.
\end{equation}
\end{proposition}

\begin{proof}
It is enough to prove \eqref{eq:L-Rodrigues-functional}, since the case \(n=1\) gives \eqref{eq:L-Pearson-Lag} because \(\hat l_1(x;\alpha)=x-\alpha-1\).

Let \(m,n\ge0\).  Using \eqref{eq:L-duality} repeatedly and \eqref{eq:Lag-L-lowering},
\[
\begin{aligned}
 \left\langle \Lop^n\Lag^\alpha,\hat l_m(x;\alpha)\right\rangle
 &=(-1)^n\left\langle\Lag^\alpha,\Lop^n\hat l_m(x;\alpha)\right\rangle                                      \\
 &=\begin{cases}
   (-1)^n\dfrac{m!}{(m-n)!}\left\langle\Lag^\alpha,\hat l_{m-n}(x;\alpha)\right\rangle,&m\ge n,\\[2mm]
   0,&m<n.
   \end{cases}
\end{aligned}
\]
Since \(\Lag^\alpha\) is normalized and annihilates \(\hat l_k(x;\alpha)\) for \(k\ge1\), the last expression is \((-1)^n n!\delta_{mn}\).  On the other hand, by \eqref{eq:Lag-norm},
\[
 \left\langle \hat l_n(x;\alpha)\Lag^\alpha,\hat l_m(x;\alpha)\right\rangle
 =n!\frac{\Gamma(n+\alpha+1)}{\Gamma(\alpha+1)}\delta_{mn}.
\]
Therefore the two sides of \eqref{eq:L-Rodrigues-functional} have the same values on the Laguerre basis.  This proves \eqref{eq:L-Rodrigues-functional}.
\end{proof}

\section{The diagonal family and the \texorpdfstring{\(\mathscr L\)}{L}-Hahn property}\label{sec:B}

We now turn to the second canonical family.  For
\[
   \beta\in\mathbb C\setminus\{0,-1,-2,\ldots\},
\]
let
\begin{equation}\label{eq:S-beta}
   \Sfun_\beta=\frac{\beta}{\beta+1}\Lag^0+\frac{1}{\beta+1}\delta.
\end{equation}
The functional \(\Sfun_\beta\) is regular, and we denote by \(\{B_n(x;\beta)\}_{n\ge0}\) its monic orthogonal polynomial sequence.  This is the diagonal family associated with \(\phi(x)=x\) and index two in the terminology of Maroni and Sfaxi.  It satisfies
\begin{equation}\label{eq:B-orthogonality}
   \langle\Sfun_\beta,B_n(x;\beta)B_m(x;\beta)\rangle=r_n(\beta)\delta_{nm},
\end{equation}
where
\begin{equation}\label{eq:B-norms}
   r_n(\beta)=(n!)^2\frac{\beta}{\beta+1}\frac{n+\beta+1}{n+\beta}.
\end{equation}
The recurrence relation is
\begin{equation}\label{eq:B-recurrence}
\begin{cases}
B_{n+1}(x;\beta)=\bigl(x-\nu_n(\beta)\bigr)B_n(x;\beta)-\lambda_n(\beta)B_{n-1}(x;\beta),\quad n\ge1,\\
B_0(x;\beta)=1,
\qquad B_1(x;\beta)=x-\nu_0(\beta),
\end{cases}
\end{equation}
with
\begin{equation}\label{eq:B-recurrence-coefficients}
   \nu_n(\beta)=2n+1-\frac{\beta}{(n+\beta)(n+\beta+1)},
   \qquad
   \lambda_n(\beta)=\frac{n^2(n+\beta-1)(n+\beta+1)}{(n+\beta)^2}.
\end{equation}
For \(\beta>0\), this is a positive definite sequence, since \(\Sfun_\beta\) is the positive measure
\[
   \frac{\beta}{\beta+1}e^{-x}\mathbf 1_{[0,\infty)}(x)\dd x
   +\frac{1}{\beta+1}\delta_0.
\]

The connection with Laguerre polynomials is the main computational device.

\begin{lemma}\label{lem:B-Laguerre-connection}
For every \(n\ge0\),
\begin{align}
   B_n(x;\beta)&=\hat l_n(x;1)+\frac{n(n+\beta+1)}{n+\beta}\hat l_{n-1}(x;1),\label{eq:B-L-connection}\\
   x\hat l_n(x;1)&=B_{n+1}(x;\beta)+\frac{(n+1)(n+\beta)}{n+\beta+1}B_n(x;\beta),\label{eq:L-B-connection}
\end{align}
where the term containing \(\hat l_{-1}\) is omitted when \(n=0\).  Consequently,
\begin{equation}\label{eq:Lag-in-B}
   \hat l_n(x;1)=\sum_{\nu=0}^n(-1)^{n-\nu}
   \frac{n!(n+\beta+1)}{\nu!(\nu+\beta+1)}B_\nu(x;\beta),
   \qquad n\ge0,
\end{equation}
and
\begin{equation}\label{eq:B-explicit}
   B_n(x;\beta)=\sum_{\nu=0}^n(-1)^{n-\nu}\binom{n}{\nu}
   \frac{n!\bigl((n+\beta+1)\nu+\beta\bigr)}{(\nu+1)!(n+\beta)}x^\nu.
\end{equation}
In particular,
\begin{equation}\label{eq:B-at-zero}
   B_n(0;\beta)=(-1)^n\frac{n!\beta}{n+\beta}.
\end{equation}
\end{lemma}

\begin{proof}
Since \(x\Sfun_\beta=\frac{\beta}{\beta+1}x\Lag^0=\frac{\beta}{\beta+1}\Lag^1\), we may expand
\[
   B_n(x;\beta)=\hat l_n(x;1)+\sum_{\nu=0}^{n-1}a_{n\nu}\hat l_\nu(x;1).
\]
For \(0\le\nu\le n-1\),
\[
 a_{n\nu}=\frac{\langle\Lag^1,B_n(x;\beta)\hat l_\nu(x;1)\rangle}{\theta_\nu(1)}
 =\frac{\beta+1}{\beta}\frac{\langle\Sfun_\beta,B_n(x;\beta)x\hat l_\nu(x;1)\rangle}{\theta_\nu(1)}.
\]
The polynomial \(x\hat l_\nu(x;1)\) has degree \(\nu+1\); hence orthogonality with respect to \(\Sfun_\beta\) gives \(a_{n\nu}=0\) for \(\nu\le n-2\).  For \(\nu=n-1\), \eqref{eq:B-norms} and \eqref{eq:Lag-norm} yield
\[
   a_{n,n-1}=\frac{\beta+1}{\beta}\frac{r_n(\beta)}{\theta_{n-1}(1)}
   =\frac{n(n+\beta+1)}{n+\beta}.
\]
This proves \eqref{eq:B-L-connection}.  The inverse connection \eqref{eq:L-B-connection} is obtained in the same way by expanding \(x\hat l_n(x;1)\) in the basis \(\{B_\nu(x;\beta)\}\).  Orthogonality leaves only the terms \(B_{n+1}\) and \(B_n\), and the coefficient of \(B_n\) is
\[
   \frac{\langle\Sfun_\beta,B_n(x;\beta)x\hat l_n(x;1)\rangle}{r_n(\beta)}
   =\frac{\beta}{\beta+1}\frac{\theta_n(1)}{r_n(\beta)}
   =\frac{(n+1)(n+\beta)}{n+\beta+1}.
\]
Formula \eqref{eq:Lag-in-B} follows by induction from \eqref{eq:B-L-connection}.  Formula \eqref{eq:B-explicit} follows by substituting the Taylor expansion of \(\hat l_n(x;1)\) into \eqref{eq:B-L-connection}; setting \(x=0\) gives \eqref{eq:B-at-zero}.
\end{proof}

We next give an intrinsic characterization of the family \(B_n(x;\beta)\).  It is formulated so that the role of the resolvent \((I-D)^{-1}\) is visible.

\begin{theorem}\label{thm:B-characterization}
Let \(\beta\notin\{0,-1,-2,\ldots\}\), and let \(\{Q_n\}_{n\ge0}\) be a MOPS with normalized regular functional \(u_0\).  The following assertions are equivalent.
\begin{enumerate}
\item \(Q_n(x)=B_n(x;\beta)\) for every \(n\ge0\).
\item There exists a monic polynomial sequence \(\{\Pi_n\}_{n\ge0}\) such that, for every \(n\ge0\),
\begin{align}
   \Pi_{n+1}(x)-D\Pi_{n+1}(x)&=Q_{n+1}(x),\label{eq:Pi-characterization-a}\\
   Q_{n+1}(0)&=-\beta\Pi_{n+1}(0).\label{eq:Pi-characterization-b}
\end{align}
\end{enumerate}
\end{theorem}

\begin{proof}
Assume first that \(Q_n=B_n(\cdot;\beta)\).  Define
\begin{equation}\label{eq:Pi-def}
   \Pi_{n+1}(x;\beta)=\hat l_{n+1}(x;0)
      +\frac{(n+1)(n+\beta+2)}{n+\beta+1}\hat l_n(x;0).
\end{equation}
Using \eqref{eq:Lag-shift} and \eqref{eq:B-L-connection}, one checks directly that
\[
   \Pi_{n+1}(x;\beta)-D\Pi_{n+1}(x;\beta)=B_{n+1}(x;\beta).
\]
Moreover, by \eqref{eq:Lag-at-zero},
\[
   \Pi_{n+1}(0;\beta)=-(-1)^{n+1}\frac{(n+1)!}{n+\beta+1},
\]
while \eqref{eq:B-at-zero} gives
\[
   B_{n+1}(0;\beta)=(-1)^{n+1}\frac{(n+1)!\beta}{n+\beta+1}.
\]
Thus \eqref{eq:Pi-characterization-b} holds.

Conversely, suppose that \eqref{eq:Pi-characterization-a}--\eqref{eq:Pi-characterization-b} hold.  By Euclidean division there is a monic polynomial \(A_n\) of degree \(n\) such that
\[
   \Pi_{n+1}(x)=xA_n(x)+\Pi_{n+1}(0)
   =xA_n(x)-\beta^{-1}Q_{n+1}(0).
\]
Substitution in \eqref{eq:Pi-characterization-a} gives the identity
\begin{equation}\label{eq:key-Pi-proof}
   Q_{n+1}(x)+\beta^{-1}Q_{n+1}(0)=(x-1)A_n(x)-xDA_n(x).
\end{equation}
Pairing \eqref{eq:key-Pi-proof} with \(\Lag^0\) and using the Pearson equation
\((x\Lag^0)' +(x-1)\Lag^0=0\), we obtain
\[
\begin{aligned}
\langle\Lag^0,Q_{n+1}\rangle+\beta^{-1}Q_{n+1}(0)
 &=\langle (x-1)\Lag^0,A_n\rangle-\langle x\Lag^0,DA_n\rangle \\
 &=\langle (x-1)\Lag^0,A_n\rangle+\langle (x\Lag^0)',A_n\rangle=0.
\end{aligned}
\]
Hence
\[
   \left\langle \Lag^0+\beta^{-1}\delta,Q_{n+1}\right\rangle=0,
   \qquad n\ge0.
\]
After normalization, the functional \(\Lag^0+\beta^{-1}\delta\) becomes \(\Sfun_\beta\).  Since a MOPS has a unique normalized orthogonality functional, \(u_0=\Sfun_\beta\), and therefore \(Q_n=B_n(x;\beta)\) for every \(n\).
\end{proof}

We can now state the main lowering result for the diagonal family.

\begin{theorem}\label{thm:B-L-Hahn}
For every \(\beta\notin\{0,-1,-2,\ldots\}\),
\begin{equation}\label{eq:B-L-lowering}
   \Lop B_{n+1}(x;\beta)=(n+1)B_n(x;\beta+1),
   \qquad n\ge0.
\end{equation}
Consequently, \(\{B_n(x;\beta)\}_{n\ge0}\) is \(\Lop\)-Hahn admissible.
Moreover,
\begin{equation}\label{eq:B-zero-shift}
   B_{n+1}(0;\beta)=-\frac{\beta}{\beta+1}(n+1)B_n(0;\beta+1),
   \qquad n\ge0.
\end{equation}
\end{theorem}

\begin{proof}
Differentiate \eqref{eq:Pi-def}.  By \eqref{eq:Lag-D},
\[
   \frac{1}{n+1}D\Pi_{n+1}(x;\beta)
   =\hat l_n(x;1)+\frac{n(n+\beta+2)}{n+\beta+1}\hat l_{n-1}(x;1).
\]
By \eqref{eq:B-L-connection}, with \(\beta\) replaced by \(\beta+1\), the right-hand side is \(B_n(x;\beta+1)\).  Since \((I-D)\Pi_{n+1}=B_{n+1}(x;\beta)\), we have \((I-D)^{-1}B_{n+1}=\Pi_{n+1}\), and therefore
\[
   \Lop B_{n+1}=D(I-D)^{-1}B_{n+1}=D\Pi_{n+1}=(n+1)B_n(x;\beta+1).
\]
This proves \eqref{eq:B-L-lowering}.  The sequence \(\{B_n(x;\beta+1)\}\) is orthogonal with respect to \(\Sfun_{\beta+1}\), hence the \(\Lop\)-derived sequence is a MOPS.  Finally \eqref{eq:B-zero-shift} is a direct consequence of \eqref{eq:B-at-zero}.
\end{proof}

\begin{remark}
The content of Theorem~\ref{thm:B-L-Hahn} is genuinely Hahn-type rather than
Appell: the lowering operation preserves orthogonality while moving the
parameter from \(\beta\) to \(\beta+1\). Thus the derived family remains
inside the same diagonal class without being required to coincide with the
original sequence.
\end{remark}

The next result gives a compact characterization of the two basic functionals
\(\Lag^0\) and \(\Sfun_\beta\) by a single boundary condition.

\begin{theorem}\label{thm:point-classification}
Let \(\{P_n\}_{n\ge0}\) be a MOPS with normalized regular functional \(u_0\).  Let
\[
   \theta\in\mathbb C\setminus\left\{0,-\frac{n}{n-1}:n\ge2\right\}.
\]
The following assertions are equivalent.
\begin{enumerate}
\item There exists \(c\in\mathbb C\) such that
\begin{equation}\label{eq:point-condition}
   P_{n+1}(c)=\theta\Lop P_{n+1}(c),
   \qquad n\ge0.
\end{equation}
\item The functional \(u_0\) is given by
\begin{equation}\label{eq:point-class-functional}
   u_0=
   \begin{cases}
      \tau_c(\Lag^0),& \theta=-1,\\[1mm]
      \tau_c\left(\Sfun_{-\theta/(\theta+1)}\right),& \theta\ne-1.
   \end{cases}
\end{equation}
\end{enumerate}
\end{theorem}

\begin{proof}
Condition \eqref{eq:point-condition} is equivalent, by \eqref{eq:L-duality}, to
\[
   \left\langle\delta_c+\theta\Lop\delta_c,P_{n+1}\right\rangle=0,
   \qquad n\ge0.
\]
The functional \(\delta_c+\theta\Lop\delta_c\) is normalized, since \(\Lop1=0\).  By uniqueness of the normalized functional annihilating \(P_n\), \(n\ge1\), we obtain
\[
   u_0=\delta_c+\theta\Lop\delta_c=\tau_c(\delta+\theta\Lop\delta).
\]
Using \eqref{eq:L0-delta}, namely \(\Lop\delta=\delta-\Lag^0\),
\[
   \delta+\theta\Lop\delta=(1+\theta)\delta-\theta\Lag^0.
\]
If \(\theta=-1\), this is \(\Lag^0\).  If \(\theta\ne-1\), setting \(\beta=-\theta/(\theta+1)\) gives precisely \(\Sfun_\beta\).  The excluded values of \(\theta\) are exactly those for which \(\beta\in\{0,-2,-3,\ldots\}\); the remaining singular value \(\beta=-1\) is not attained by finite \(\theta\).  The converse is obtained by reversing the argument, using \eqref{eq:Lag-L-lowering} for \(\Lag^0\) and \eqref{eq:B-zero-shift} for \(\Sfun_\beta\).
\end{proof}

\begin{remark}
Theorem~\ref{thm:point-classification} is a rigidity statement. No Appell
assumption and no a priori Laguerre or diagonal representation is imposed:
the scalar relation \eqref{eq:point-condition}, required at one point for all
degrees, determines the normalized orthogonality functional completely.
\end{remark}

The lowering theorem also produces useful finite structure relations. We
state them with the convention \(B_{-1}(x;\beta)=0\).

\begin{proposition}\label{prop:B-structure-relations}
Let
\[
   B_n^{\langle1\rangle}(x;\beta)=\frac{1}{n+1}\Lop B_{n+1}(x;\beta).
\]
Then \(B_n^{\langle1\rangle}(x;\beta)=B_n(x;\beta+1)\) and, for \(n\ge0\),
\begin{align}
   xB_n(x;\beta)&=B_{n+1}^{\langle1\rangle}(x;\beta)+a_nB_n^{\langle1\rangle}(x;\beta)
      +b_nB_{n-1}^{\langle1\rangle}(x;\beta),\label{eq:structure-1}\\
   xB_n^{\langle1\rangle}(x;\beta)&=B_{n+1}(x;\beta)+c_nB_n(x;\beta)+d_nB_{n-1}(x;\beta),\label{eq:structure-2}\\
   B_n(x;\beta)+t_nB_{n-1}(x;\beta)&=B_n^{\langle1\rangle}(x;\beta)+s_nB_{n-1}^{\langle1\rangle}(x;\beta),\label{eq:structure-3}
\end{align}
where
\begin{align*}
   a_n&=(n+\beta+1)\left(\frac{n+1}{n+\beta+2}+\frac{n}{n+\beta}\right),
   & b_n&=n^2,\\
   c_n&=\frac{(n+1)(n+\beta)+n(n+\beta+2)}{n+\beta+1},
   & d_n&=\frac{n^2(n+\beta+2)(n+\beta-1)}{(n+\beta)(n+\beta+1)},\\
   t_n&=\frac{n(n+\beta-1)}{n+\beta},
   & s_n&=\frac{n(n+\beta)}{n+\beta+1}.
\end{align*}
\end{proposition}

\begin{proof}
The identity \(B_n^{\langle1\rangle}(x;\beta)=B_n(x;\beta+1)\) is \eqref{eq:B-L-lowering}.  Formula \eqref{eq:structure-1} follows by multiplying \eqref{eq:B-L-connection} by \(x\) and using \eqref{eq:L-B-connection} with \(\beta\) replaced by \(\beta+1\).  Formula \eqref{eq:structure-2} follows similarly by writing \(B_n^{\langle1\rangle}(x;\beta)=B_n(x;\beta+1)\), using \eqref{eq:B-L-connection} with parameter \(\beta+1\), and then using \eqref{eq:L-B-connection} with parameter \(\beta\).  Finally, \eqref{eq:structure-3} is obtained by expressing both sides in the Laguerre basis \(\{\hat l_j(x;1)\}\).  The coefficients of \(\hat l_n\), \(\hat l_{n-1}\) and \(\hat l_{n-2}\) coincide, respectively, and no other terms occur.
\end{proof}

The functional \(\Sfun_\beta\) is semiclassical of class one in the ordinary differential sense.  This can be read directly from the following characterization.

\begin{proposition}\label{prop:B-differential-characterization}
Let \(\{P_n\}_{n\ge0}\) be a MOPS with normalized regular functional \(u_0\), and assume \((u_0)_1\ne1\).  The following statements are equivalent.
\begin{enumerate}
\item There exist complex coefficients \(\lambda_{n+1,\nu}\), \(\nu=n-1,n,n+1\), with \(\lambda_{n+1,n+1}=1\) and \(\lambda_{2,0}=0\), such that
\begin{equation}\label{eq:differential-structure-general}
   x^2P_n''(x)-x(x-2)P_n'(x)
   =-n\sum_{\nu=n-1}^{n+1}\lambda_{n+1,\nu}P_\nu(x),
   \qquad n\ge1.
\end{equation}
\item The functional \(u_0\) satisfies
\begin{equation}\label{eq:Sbeta-Pearson-D}
   (x^2u_0)' +x(x-2)u_0=0.
\end{equation}
\item One has
\begin{equation}\label{eq:u0-Sbeta-moment}
   u_0=\Sfun_\beta,
   \qquad
   \beta=\frac{(u_0)_1}{1-(u_0)_1}.
\end{equation}
\end{enumerate}
For \(P_n=B_n(x;\beta)\), the coefficients in \eqref{eq:differential-structure-general} are
\begin{align}\label{eq:B-differential-equation}
 x^2B_n''(x;\beta)-x(x-2)B_n'(x;\beta)
 &=-n\sum_{\nu=n-1}^{n+1}\lambda_{n+1,\nu}(\beta)B_\nu(x;\beta),
 \qquad n\ge1,
\end{align}
with
\begin{align*}
\lambda_{n+1,n+1}(\beta)&=1,\\
\lambda_{n+1,n}(\beta)&=\frac{(n+1)(n+\beta)}{n+\beta+1}
      +\frac{(n-1)(n+\beta+1)}{n+\beta},\\
\lambda_{n+1,n-1}(\beta)&=\frac{n(n-1)(n+\beta-1)(n+\beta+1)}{(n+\beta)^2}.
\end{align*}
\end{proposition}

\begin{proof}
Assume \eqref{eq:differential-structure-general}.  Pairing with \(u_0\), the right-hand side vanishes for every \(n\ge1\); for \(n=1\) this is exactly the role of the condition \(\lambda_{2,0}=0\).  Therefore
\[
   \langle u_0,x^2P_n''-x(x-2)P_n'\rangle=0,
   \qquad n\ge1.
\]
By duality this is
\[
   \left\langle (x^2u_0)' +x(x-2)u_0,P_n'\right\rangle=0,
   \qquad n\ge1.
\]
The sequence \(\{P_n'\}_{n\ge1}\) is a basis of \(\Pcal\), hence \eqref{eq:Sbeta-Pearson-D} follows.

Conversely, if \eqref{eq:Sbeta-Pearson-D} holds, then pairing it with \(x^k\) gives
\[
   (u_0)_{k+2}=(k+2)(u_0)_{k+1},
   \qquad k\ge0.
\]
Thus, with \(a=(u_0)_1\),
\[
   (u_0)_0=1,
   \qquad
   (u_0)_n=a n!,\quad n\ge1.
\]
If \(a=0\), then \(u_0=\delta\), which is not regular. Since \(a\ne1\) by
assumption,
\[
   u_0=a\Lag^0+(1-a)\delta=\Sfun_{a/(1-a)}.
\]
Thus (2) implies (3). Conversely, if \(u_0=\Sfun_\beta\), uniqueness of the
monic orthogonal sequence gives \(P_n=B_n(x;\beta)\). Applying the Laguerre
equation \eqref{eq:Lag-ODE} with \(\alpha=1\) to the connection formula
\eqref{eq:B-L-connection}, multiplying by \(x\), and using
\eqref{eq:L-B-connection} to return to the \(B\)-basis yields
\eqref{eq:B-differential-equation}; in particular, (1) holds. This completes
the equivalence and gives the displayed coefficients.
\end{proof}

We close the section with the generating function.  Let
\[
   \Pi_\beta(x,t)=\sum_{n\ge0}\frac{B_n(x;\beta)}{n!}t^n.
\]

\begin{proposition}\label{prop:B-generating-function}
For \(\beta>-1\), \(\beta\ne0\), and \(|t|<1\),
\[
   \Pi_\beta(x,t)
   =
   \frac{\exp\left(\frac{xt}{1+t}\right)}{1+t}
   +
   t^{-\beta}
   \int_0^t
   y^\beta
   \frac{\exp\left(\frac{xy}{1+y}\right)}{(1+y)^2}
   \,dy.
\]
Moreover,
\[
   \Lop_x\Pi_\beta(x,t)=t\Pi_{\beta+1}(x,t).
\]
\end{proposition}

\begin{proof}
The classical exponential generating function of the monic Laguerre sequence is
\[
   F(x,t)
   :=
   \sum_{n\ge0}\hat l_n(x;1)\frac{t^n}{n!}
   =
   \frac{\exp\left(\frac{xt}{1+t}\right)}{(1+t)^2}.
\]
Using the connection formula
\[
   B_n(x;\beta)
   =
   \hat l_n(x;1)
   +
   \frac{n(n+\beta+1)}{n+\beta}\hat l_{n-1}(x;1),
\]
and putting \(m=n-1\) in the second term, we obtain
\[
\begin{aligned}
   \Pi_\beta(x,t)
   &=
   F(x,t)
   +
   t\sum_{m\ge0}
   \frac{m+\beta+2}{m+\beta+1}
   \hat l_m(x;1)\frac{t^m}{m!}  \\
   &=
   (1+t)F(x,t)
   +
   t\sum_{m\ge0}
   \frac{1}{m+\beta+1}
   \hat l_m(x;1)\frac{t^m}{m!}.
\end{aligned}
\]
For \(\beta>-1\), the last sum has the integral representation
\[
   \sum_{m\ge0}
   \frac{1}{m+\beta+1}
   \hat l_m(x;1)\frac{t^m}{m!}
   =
   t^{-\beta-1}
   \int_0^t
   y^\beta
   \frac{\exp\left(\frac{xy}{1+y}\right)}{(1+y)^2}
   \,dy.
\]
Therefore
\[
   \Pi_\beta(x,t)
   =
   \frac{\exp\left(\frac{xt}{1+t}\right)}{1+t}
   +
   t^{-\beta}
   \int_0^t
   y^\beta
   \frac{\exp\left(\frac{xy}{1+y}\right)}{(1+y)^2}
   \,dy.
\]
This proves the generating function. Finally, applying \(\Lop_x\) term by term
and using Theorem~\ref{thm:B-L-Hahn}, we get
\[
   \Lop_x\Pi_\beta(x,t)=t\Pi_{\beta+1}(x,t).
\]
\end{proof}

\section{Concluding remarks}

The resolvent delta operator
\[
   \Lop=D(I-D)^{-1}
\]
exhibits two distinct and rigorously identifiable mechanisms of orthogonality
under lowering. In the Appell regime the situation is completely rigid:
Theorem~\ref{thm:Lag-Appell-characterization} shows that the translated
Laguerre sequences are exactly the monic orthogonal \(\Lop\)-Appell sequences.
Beyond the Appell regime, Theorem~\ref{thm:B-L-Hahn} shows that the diagonal
family is stable under \(\Lop\) through the exact parameter shift
\(\beta\mapsto\beta+1\).

The boundary characterization of
Theorem~\ref{thm:point-classification} provides a complementary form of
rigidity. A single pointwise compatibility relation between evaluation and
\(\Lop\)-lowering determines the normalized orthogonality functional and
recovers precisely the translated Laguerre functional at parameter zero or a
translated diagonal functional. The connection formulae, finite structure
relations, semiclassical Pearson equation and generating function then show
that this phenomenon is simultaneously algebraic, functional and
differential.

These results give a self-contained Hahn-type theory for the two mechanisms
treated in the paper.
They also provide a concrete bridge between finite operator calculus and the
semiclassical theory of orthogonal polynomials: the resolvent
\((I-D)^{-1}\) converts classical derivative identities into exact
orthogonality-preserving lowering relations. In this sense, the diagonal
family is not merely a point-mass perturbation of Laguerre orthogonality; it is
an operator-compatible deformation whose parameter evolution is encoded
directly by \(\Lop\).

\bibliographystyle{plainnat}
\bibliography{Hahn_Resolvent_References}

\section*{Statements and Declarations}

\subsection*{Funding}
The first author acknowledges support from the Étoiles Montantes en
Pays de la Loire project, funded by the Pays de la Loire Region.
The second author acknowledges institutional support from the Faculty
of Sciences of Gabès and the Faculty of Sciences of Sfax.

\subsection*{Competing Interests}
The authors have no relevant financial or non-financial interests to
disclose.

\subsection*{Author Contributions}
Both authors contributed to the conception and development of the
results, the preparation of the manuscript, and its revision.
Both authors read and approved the final manuscript.

\subsection*{Data Availability}
No datasets were generated or analysed during the current study.

\end{document}